\documentclass[11pt, a4paper]{amsart}
\usepackage{amsmath}
\newtheorem{theorem}{Theorem}[section]

\newtheorem{lemma}{Lemma}[section]

\newtheorem{corollary}{Corollary}[section]

\usepackage[colorlinks=true,citecolor=cyan,backref=page]{hyperref}

\usepackage[utf8]{inputenc}

\usepackage{tikz}
\usetikzlibrary{patterns}
\usetikzlibrary{fadings}
\tikzfading[name=fade out, inner color=transparent!0, outer color=transparent!90]
\usepackage[labelfont=normalfont,labelformat=simple]{subcaption}
\usepackage{graphicx}
\usepackage[vcentermath]{youngtab}
\usepackage{ytableau}

\usepackage{enumerate}

\usepackage{todonotes}

\newcommand{\bw}{\mathrm{bw}}

\title{Characterizations of perfectly clustering words}

\author{Mélodie Lapointe}
\address{Mélodie Lapointe,
D\'epartement de math\'ematiques et de statistique, Universit\'e de Moncton}
\email{melodie.lapointe@umoncton.ca}

\author{Christophe Reutenauer}
\address{Christophe Reutenauer,
D\'epartement de math\'ematiques, Universit\'e du Qu\'ebec \`a Montr\'eal}
\email{Reutenauer.Christophe@uqam.ca}

\keywords{perfectly clustering word, factorization, palindrome, Burrows-Wheeler transform, symmetric discrete interval exchange}

\date{\today}

\begin{document}

\maketitle

 \tableofcontents

\begin{abstract} 
    Perfectly clustering words are one of many possible generalizations of Christoffel words. In this article, we propose a factorization of a perfectly clustering word on a $n$ letters alphabet into a product of $n-1$ palindromes with a letter between each of them. This factorization allows us to generalize two combinatorial characterization of Christoffel words due to Pirillo (1999) and  de Luca and Mignosi (1994).
\end{abstract}

\section{Introduction}

Christoffel words and related families of binary words like standard 
words and central words have been extensively studied, resulting in many 
equivalent but seemingly unrelated definitions of these sets of words. 
Several generalizations have been studied in the literature, for larger 
alphabets, or in larger dimension \cite{CMR, J, LR, MR, Pa}. However, the 
equivalence between the different definitions of Christoffel words does not extend 
in general. 

In this article, we aim to study properties of perfectly clustering words, a generalisation of Christoffel words proposed in \cite{SP,FZ}. A word $w$ is perfectly clustering if its Burrows-Wheeler transform is a weakly decreasing word. In \cite{MRS}, Mantaci, Restivo and Sciortino showed that Christoffel words and their conjugates are exactly the binary words whose Burrows-Wheeler transform is a weakly decreasing word. Later, Ferenczi and Zamboni \cite{FZ} related these words with symmetric interval exchange transformations. Moreover, they also are primitive elements of the free group~\cite{L2021}, and product of two palindromes \cite{SP}. Thus perfecly clustering words share several properties of Christoffel words. 

Our main result (Theorem \ref{char1}) uses what we call the {\em special factorization} of perfectly clustering Lyndon word, in order two give two characterizations of these words. This result generalizes the following characterizations of Christoffel words: 

- a binary word $amb$ is a Christoffel word if and only if the word $amb$ is a product of two palindromes and $m$ is also a palindrome (de Luca and Mignosi \cite{dLM}).

- a binary word $amb$ is a Christoffel word if and only if the words $amb$ and $bma$ are conjugates (Pirillo \cite{P}). 
%In this process, we also obtain generalization of result on Christoffel word presented in~\cite{dLM}: 

The article is structured as follows. In Section 2, we recall known 
results and definitions. In Section 3, we introduced the special 
factorisation and characterize perfectly clustering word with it. In 
Section 4, we prove several lemmas that are necessary for our 
characterization. In Section 6, we present the proof of it. In Section 
5, we give some results on the structure of the Burrows-Wheeler matrix of a perfectly clustering word. 

The authors thank 
Antonio Restivo for useful mail exchanges.

This work was partially supported by NSERC, Canada. 

\section{Perfectly clustering words}\label{basic}

\subsection{Words}

Let $A = \{a_1,a_2,\dots, a_k\}$ be a totally ordered alphabet, where $a_1 < a_2 < \dots < a_k$.
Let $A^*$ denote the free monoid generated by $A$ and $F(A)$ denote the free group on $A$. An element in 
$F(A)$ which lies actually in $A^*$ is called {\it positive}. 

The length of $w = b_1\cdots b_n$ (with $b_i\in A$), denoted by $|w|$, is $n$.
The number of occurrences of a letter $a$ in $w$ is denoted by $|w|_a$.
The {\em commutative image} of $w$ is the integer vector $(|w|_{a_1},\dots,|w|_{a_k})$.
The function $Alph$ is defined by $Alph(w) = \{x \in A \mid |w|_x \geq 1\}$. 
%If $w \in A^*$, then $Alph(w) \subseteq A$.

A word $w$ is called {\em primitive} if it is not the power of another word, that is: for any word $z$ such that $w = z^n$, one has $n = 1$.
The \emph{conjugates} of a word $w$, as above, are the words $b_i \dots b_n b_1 \dots b_{i-1}$. In other words, two 
words $u,v\in A^*$ are {\it conjugate} if for some words $x,y\in A^*$, one has $u=xy,v=yx$. The {\em conjugation class} of a word is the set of its conjugates.
If a word $w$ is primitive, then it has exactly $n$ distinct conjugates.

A {\it palindrome} in the free group $F(A)$ is an element fixed by the 
unique anti-automorphism $g\mapsto \tilde g$ of $F(A)$ such that $\tilde 
a=a$ for any $a\in A$. We call $\tilde g$ the {\it reversal} of $g$.
It is well-known that a word $w$ in $A^*$ is conjugate to its reversal if and only if $w$ 
is a product of two palindromes (see \cite{BHNR}).

The {\em lexicographic} order is an extension of 
the total order on $A$ defined as follows: if $u,v \in A^*$, we have $u 
< v$ if either $u$ is a proper prefix of $v$, or $u = rxs$ and $v = ryt$ 
such that $x < y$ and $x,y \in A$ and $r,s,t \in A^*$.
A word $w$ is called a {\em Lyndon word} if it is a primitive word, and 
it is the minimal word in lexicographic order among its conjugates.

\subsection{Perfectly Clustering Words}\label{pcw}

Let $v$ be a primitive word of length $n$ on the alphabet $A$.
Let $v_1 < v_2 < \dots < v_n$ be its conjugates, lexicographically ordered.
Let $l_i$ be the last letter of the word $v_i$ for $1\leq i \leq n$.
The {\em Burrows-Wheeler transform} of the word $v$, denoted by $\bw(v)$, is the word $l_1 
\cdots l_n$.
Define the {\em Burrows-Wheeler matrix} of $v$ to be the matrix, with rows corresponding 
bijectively 
to the conjugates of $w$, lexicographically ordered, and with row-elements being
the letters of the
corresponding conjugate; then $bw(v)$ is the last column of this 
matrix. For example, the Burrows-Wheeler transform of the word $apartment$ is 
$\bw(apartment) =tpmteaanr$ (see Figure~\ref{fig:bwt}). 
It follows from the definition that two words $u$ and $v$ are conjugates if and only if 
$\bw(u) = \bw(v)$ (see~\cite[Proposition 1]{MRS}).
\begin{figure}
\setcounter{MaxMatrixCols}{20}
\begin{align*}
    \begin{matrix}
            a & p & a & r & t & m & e & n & 
            \smash{\color{black}\fbox{\color{black}\rule[-105pt]{0pt}{2pt}$t$}}\\
            a & r & t & m & e & n & t & a & p  \\
            e & n & t & a & p & a & r & t & m  \\
            m & e & n & t & a & p & a & r & t  \\
            n & t & a & p & a & r & t & m & e  \\
            p & a & r & t & m & e & n & t & a  \\
            r & t & m & e & n & t & a & p & a  \\
            t & a & p & a & r & t & m & e & n  \\
            t & m & e & n & t & a & p & a & r  \\
    \end{matrix}
\end{align*}
\caption{The Burrows-Wheeler matrix of the word $apartment$ sorted in lexicographic order. The last column is its Burrows-Wheeler transform}
\label{fig:bwt}
\end{figure}

Following \cite{FZ}, we say that a primitive word $v$ is {\em $\pi$-clustering} if $$\bw(v) = a_{\pi(1)}^{|w|_{a_{\pi(1)}}} \dots a_{\pi(r)}^{|w|_{a_{\pi(k)}}},$$ where $\pi$
is a permutation on $\{1,\dots, k\}$.
For example, the word $aluminium = a_1 a_3 a_6 a_4 a_2 a_5 a_2 a_6 a_4$ is $451623$-clustering, since $\bw(aluminium) = mmnauuiil$.
A word is {\em perfectly clustering} if the permutation $\pi$ is the symmetric permutation, i.e. $\pi(i) = k - i + 1$, for all $i \in \{1,\dots,k\}$.
It was proved by Sabrina Mantaci, Antonio Restivo and Marinella Sciortino that the perfectly clustering words on a two-letter alphabet are the Christoffel words 
and their conjugates (see~\cite[Theorem 9]{MRS}; see also \cite[Theorem 15.2.1]{R}).
If a primitive word is perfectly clustering, then all its conjugates are. Consequently, there is no loss of generality to study perfectly clustering Lyndon words.

For each letter 
$\ell\in A$, we define, 
following \cite{La}, 
two automorphisms $\lambda_\ell$ and $\rho_\ell$ of the free group $F(A)$ by: 
\begin{center}
    \begin{tabular}{ccc}
        $\lambda_{\ell}(a) = \begin{cases}
            
            a\ell^{-1}, & \text{ if } a < \ell ; \\
            a, & \text{ if } a = \ell ; \\
            \ell a, & \text{ if } a > \ell ; \\
        \end{cases}$
        & and &
        $\rho_{\ell}(a) = \begin{cases}
            
            a\ell, & \text{ if } a < \ell; \\
            a, & \text{ if } a = \ell ;\\
            \ell^{-1} a, & \text{ if } a > \ell ;\\
        \end{cases}$
    \end{tabular}
\end{center}
for any $a\in A$.

It was proved by the first author that for each perfectly clustering word $w\in A^*$ of length at least 3, there exists a shorter 
perfectly clustering word $u\in A^*$ and an automorphism $f=\lambda_\ell$ or $\rho_\ell$ such that $w=f(u)$ (see \cite[Chapter 4]{La}). One 
may also assume that either a) $Alph(w)=Alph(u)=A$, or b) $A=Alph(w)=Alph(u)\cup \ell$ and $\ell\notin Alph(u)$.

Since $\rho_\ell$ leaves invariant the number of letters in a word, except letter $\ell$, the fact that $w$ is longer than $u$ implies that $|w|_\ell>|u|_\ell$, that is, $$
|w|_\ell = |u|_\ell + \sum_{a < \ell}|u|_a-\sum_{a>\ell}|u|_a>|u|_\ell$$ (by the special form of $\rho_\ell$); therefore $
\sum_{a<\ell}|u|_a>\sum_{a>\ell}|u|_a$.

It is also shown in \cite[Lemma 4.6]{La} that for each word $u\in A^*$, $\lambda_\ell(u)$ (resp. $\rho_\ell(u)$) $\in A^*$ (that is, is positive) if and only if each 
letter $<\ell$ (resp. $>\ell$) in $w$ is followed (resp. preceded) in $w$ by a letter $\geq \ell$ (resp. $\leq \ell$).

\subsection{Symmetric discrete interval exchanges}\label{interval-exchange}

We follow \cite{FZ}. Let $n \geq 1$ be an integer. Let $(c_1,\ldots,c_k)$ be a {\it composition} of $n$, 
that is, an $k$-tuple of positive integers whose sum is $n$. We decompose in two ways the 
interval $[n]=\{1,2,\ldots,n\}$ into  
intervals: the intervals $I_1,\ldots,I_k$ (resp. $J_1,\ldots,J_k$) are defined by the 
condition that they are consecutive and that $\vert I_h\vert 
=c_h$ (resp. $\vert J_h\vert=c_{k+1-h}$). Denote by $S_n$ the group of permutations of $[n]$. We define the permutation $\sigma\in S_n$
by the condition that it sends increasingly each interval $I_h$ onto the interval $J_{k+1-h}$. We call such a permutation a {\em symmetric 
discrete interval exchange}\footnote{The word ``symmetric" refers to the fact that the 
intervals are exchanged according to the central symmetry to the set $
\{1,2,\ldots,k\}$, that is, the mapping $h\mapsto k+1-h$.}, and it is said to 
be {\em associated with the composition} $(c_1,\ldots,c_k)$.

It follows from this definition that for each $i<j$ in $[n]$ such that $i,j$ are in the 
same interval $I_h$, one has $\sigma(i)<\sigma(j)$.

Note that the restriction of $\sigma$ to $I_h$ is the translation by some integer $t_h$; 
since its image is $J_{k+1-h}$, one has $t_h=\max(J_{k+1-h})-\max(I_h)$. One has 
$\max(I_h)=c_1+\cdots+c_h$ and $\max(J_{k+1-h})=c_k+\cdots+c_h$, hence
$t_h=\sum_{i>h}c_i-\sum_{i<h}c_i$. Thus $\sigma$ is completely defined by:
\begin{equation}\label{local-translations}
\forall h=1,\ldots,k, \forall x\in I_h,\, \sigma(x)=x+t_h.
\end{equation}

We call $\sigma$ a {\it circular} symmetric discrete interval exchange if $\sigma$ is a circular permutation, that is, has only one 
cycle. 

A {\it word encoding} of a circular symmetric discrete interval exchange $\sigma$ is one of the words obtained by replacing in 
one of the cycle forms of $\sigma$ each number $i$ by the letter $a_j$ such that $i\in I_j$.  

\begin{theorem}\label{2conditions} The following conditions are equivalent, for a primitive word $w$ on a totally ordered alphabet:

(i) $w$ is primitive perfectly clustering word;

(ii) $w$ is a word encoding of some circular symmetric interval exchange.
\end{theorem}

This result is due to Sébastien Ferenczi and Luca Zamboni \cite[Theorem 4]{FZ}. A brief history around this result is given in Section~\ref{sec:history}.

Denote by $m_h$ the minimum of the interval $I_h$. For later use, we prove two lemmas.
\begin{lemma}\label{mhmh+1} For any $h=1,\ldots,k-1$, one has $\sigma(m_h)=m_{h+1}$ if and only if $\sum_{i>h}c_i-\sum_{i\leq h}c_i=0$.
\end{lemma}

\begin{proof} By definition of the intervals $I_h$, one has $m_h=\sum_{i<h}c_i+1$.
Hence, $\sigma(m_h)=m_h+t_h=\sum_{i<h}c_i+1+\sum_{i>h}c_i-\sum_{i<h}c_i=1+\sum_{i>h}c_i$. This is equal to $m_{h+1}=\sum_{i\leq h}c_i+1$ if and only if the equality in the lemma holds.
\end{proof}

\begin{lemma} (\cite[Proof of Theorem 4]{FZ} )\label{order-conjugates}
Let $\sigma$ be a circular symmetric discrete interval exchange, with the notations above. For each cyclic representation $\gamma_r=(r,\sigma(r),\cdots,\sigma^{n-1}(r))$ of $\sigma$, denote by $w_r$ the word obtained by replacing in $\gamma_r$ each number $i$ by $a_j$ if $i\in I_j$. Then
$$w_r<_{lex}w_s \Leftrightarrow r<s.
$$
\end{lemma}

\begin{proof}
Suppose that $r<s$. Let $p\leq n$ be maximal such that the prefixes of length $p$ of $w_r$ and $w_s$ coincide. Then for each 
$i=1,\ldots,p$, the $i$-th letters of $w_r$ and $w_s$ are equal, hence $\sigma^{i-1}(r)$ and $\sigma^{i-1}(s)$ are in the same interval $I_j$. It follows recursively from the observation before (\ref{local-translations}) that $
\sigma^{p}(r)<\sigma^p(s)$. Suppose that $p<n$; since $p$ is maximal, the previous two numbers are not in the same interval, so that 
they are respectively in two intervals $I_h,I_j$ with $h<j$ (the intervals are successive); hence
the 
$(p+1)$-th letter of $w_r$ (obtained by replacing $\sigma^p(r)$ according to the rule in the statement)  is $a_h$, which is strictly smaller than $a_j$, which is the $(p+1)$-th letter of $w_s$ and therefore $w_r<_{lex} w_s$. Suppose now by contradiction that $p=n$; then we 
have $\sigma^{i-1}(r)<\sigma^{i-1}(s)$ for any $i=1,\ldots,n$, and this is not possible since $\sigma$ is circular so that some $\sigma^{i-1}(s)$ 
is equal to 1.
This ends the proof of the equivalence, since both orders are total.
\end{proof}

This lemma has several consequences: if we identify the set $\mathcal C$ of conjugates of $w$, lexicographically ordered, with $[n]$, under the unique order 
isomorphism between the two sets, then $\sigma$ is identified with the {\em conjugator} $C$, which is the mapping 
sending each nonempty word $au$ onto $ua$ ($a\in A,u\in A^*$). Moreover the interval $I_s$ is identified with the subset of $\mathcal C$ of words beginning by $a_s$. And $m_s$ is identified with the smallest conjugate of $w$ beginning by $a_s$.

\subsection{Brief history of Theorem \ref{2conditions}}\label{sec:history}

Continuous interval exchanges have a long history, beginning in 1966 by Osedelets \cite{O}. But discrete interval exchanges appeared more recently, in 2013, in 
the article of Ferenczi and Zamboni \cite{FZ}. They prove a more general theorem that Theorem \ref{2conditions}; the latter is obtained from theirs by taking $\pi$
to be the longest permutation in $S_k$, that is $\pi(i)=k+1-i$. However, for a two-letter alphabet, the result was given implicitly in \cite[Section 3 and Theorem 9]{MRS} (2002). Moreover, for a general alphabet, in \cite{SP} (2008) the function $\omega$ page 13 is a symmetric discrete interval exchange, as shows the  
formula there, similar to (\ref{local-translations}).
%; the authors show, among other results, that perfectly clustering words are product of two palindromes; and they characterize completely the perfectly clustering words on a three-letter alphabet:
It turns out that, for perfectly clustering words $w$, the intervals $I_h$ and $J_h$ defining the associated symmetric discrete interval exchange may be read 
directly on the first and last columns of the Burrows-Wheeler matrix of $w$. This permutation extends for any word, not necessarily perfectly clustering, and was 
observed in the original preprint of Burrows-Wheeler \cite[page 4]{BW} (1994). Finally, note that the Burrows-Wheeler transform is a particular case, discovered 
independently, of a bijection due to Gessel and the second author \cite{GR} (1993).

\section{A characterization of perfectly clustering words}

Let $w$ be a word on a totally alphabet $A$. We call {\it special factorization of} $w$ a factorization of the form
\begin{equation}\label{special}
w=a_1\pi_1a_2\pi_2\cdots \pi_{k-1}a_k,
\end{equation}
where $Alph(w)=\{a_1<a_2<\cdots<a_k\}$ and $\pi_1,\ldots,\pi_{k-1}\in A^*$. If $w$ satisfies this equality, we set $$W=a_k\pi_{k-1}\cdots \pi_2a_2\pi_1a_1.$$ We call this special factorization {\it palindromic} if each word $\pi_i$ is a palindrome; in this case $W=\tilde w$.

\begin{theorem}\label{char1} The following conditions are equivalent, for a primitive word $w$ on the totally ordered finite alphabet $A$:

(i) $w$ is a perfectly clustering Lyndon word;

(ii) $w$ is a product of two palindromes and $w$ has a palindromic special factorization;

(iii) $w$ has a special factorization (\ref{special}) such that $w$ is conjugate to $W$.
\end{theorem}

This theorem generalizes known results on perfectly clustering Lyndon words on two letters, which are the Christoffel words, according to \cite{MRS}. The equivalence of (i) and (ii) is a generalization of a result of Aldo de Luca and Filippo Mignosi, in \cite[Proposition 8]{dLM} (see also \cite[Theorem 2.2.4]{L} 
or \cite[Theorem 12.2.10]{R}). The equivalence of (i) and (iii) is a generalization of a result of Giuseppe Pirillo \cite{P} (see also \cite[Theorem 15.2.5]{R}).

The theorem will be proved in Section \ref{proof}.
\section{Several lemmas}

\begin{lemma}\label{pal} If $g\in F(A)$ is a palindrome, then so is $\lambda_\ell(g)\ell$ (resp. $\ell \rho_\ell(g))$.
\end{lemma}

\begin{proof} We prove it for $\rho_\ell$, the proof for $\lambda_\ell$ being similar.
If $g\in\{1\}\cup A\cup A^{-1}$, this is easily verified. If $g$ is a palindrome in $F(A)$, not of the previous form, then $g=xhx$ 
with $x\in A\cup A^{-1}$, and $h$ a palindrome, shorter than $g$. By induction, $\ell\rho_\ell(h) $ is a palindrome, hence $\rho_\ell(h)
\ell^{-1}$ too. We have $\ell\rho_\ell(g)=\ell\rho_\ell(x)\rho_\ell(h)\rho_\ell(x)$. 

Suppose first that $x=a\in A$. If $a<\ell$, then $\ell\rho_\ell(g)=\ell a\ell\rho_\ell(h)a\ell$; if $a=\ell$, then $\ell\rho_\ell(g)=\ell\ell\rho_\ell(h)\ell$; if 
$a>\ell$, then $\ell\rho_\ell(g)=\ell\ell^{-1}a\rho_\ell(h)\ell^{-1}a=a\rho(h)\ell^{-1}a$. 
%In each case, $\ell\rho(g)$ is a palindrome.

Suppose now that $x=a^{-1}$ and $a\in A$. If $a<\ell$, then $\ell\rho_\ell(g)=\ell\ell^{-1}a^{-1}\rho_\ell(h)\ell^{-1}a^{-1}=a^{-1}\rho_\ell(h)
\ell^{-1}a^{-1}$. If 
$a=\ell$, then $\ell\rho_\ell(g)=\ell\ell^{-1}\rho_\ell(h)\ell^{-1}=\rho_\ell(h)\ell^{-1}$; if $a>\ell$, then $\ell\rho_\ell(g)=\ell 
a^{-1}\ell\rho_\ell(h)a^{-1}\ell$. 

In each 
case $\ell\rho_\ell(g)$ is a palindrome.
\end{proof}

\begin{lemma}\label{positive-pal} Let $u\in A^*,\, b,\ell\in A$, with $b>\ell$. Suppose that $\rho_\ell(u)$ is positive, and that $u$ has a factor $\pi b$, where $\pi$ is a 
nonempty palindrome. Then $\rho_\ell(\pi)\ell^{-1}$ is a positive palindrome.
\end{lemma}

\begin{proof} We know that each letter $>\ell$ in $u$ is preceded by a letter $\leq \ell$. Hence the last letter of $\pi$ is $a\leq \ell$, and equal to its first letter: $\pi=ava$, or $\pi=a$. If $\pi=a$, then $\rho_\ell(\pi)\ell^{-1}=a$ if $a<\ell$, and $\rho_\ell(\pi)\ell^{-1}=1$ if $a=\ell$, which settles this case.

Suppose now that $\pi=ava$.
Then each letter $>\ell$ in $av$ is preceded by a letter $\leq \ell$; therefore $\rho_\ell(av)$ is positive. If $a<\ell$, then $\rho_\ell(\pi)\ell^{-1}
=\rho_\ell(av)a\ell\ell^{-1}=\rho_\ell(av)a$ is positive; if $a=\ell$, $\rho_\ell(\pi)\ell^{-1}=\rho_\ell(av)\ell\ell^{-1}=\rho_\ell(av)$ is positive. Moreover, $\rho_\ell(\pi)\ell^{-1}$ is a palindrome by Lemma \ref{pal}.
\end{proof}

The easy proofs of the two following lemmas are left to the reader.

\begin{lemma}\label{next} Let $E$ be a totally ordered finite set, with some element $M$. Suppose that there exists an injective function $\nu:E\setminus M\to E$ such that $\forall x\in E\setminus M, x<\nu(x)$. Then $M=\max(E)$ and $\nu$ is the next element function of $E$ (that is $\nu(x)=\min\{y\in E,y>x\}$).
\end{lemma}

\begin{lemma}\label{<} Let $w$ be a word with a special factorization as in (\ref{special}). Consider $i=1,\ldots,k-1$ and some factorization $\pi_i=uv$. Then
$$
va_i\cdots \pi_1a_1a_k\pi_{k-1}\cdots a_{i+1}u < va_{i+1}\cdots \pi_{k-1}a_ka_1\pi_1\cdots a_iu ,
$$
the left-hand side word is conjugate to $W$, and the right-hand side word is conjugate to $w$.
\end{lemma}

\begin{lemma}\label{consecutive} Let $w$ be a primitive word with a special factorization (\ref{special}), such that $w$ is conjugate to $W$. Consider two consecutive rows of its Burrows-Wheeler matrix, viewed as two words $w',w''$. Then $$w'=va_i\cdots \pi_1a_1a_k\pi_{k-1}\cdots a_{i+1}u$$ and $$w''=va_{i+1}\cdots \pi_{k-1}a_ka_1\pi_1\cdots a_iu,$$ for some $i=1,\ldots,k-1$, and some factorization $\pi_i=uv$. Moreover, $w$ is the smallest, and $W$ the largest, element in their conjugation class.
\end{lemma}

\begin{proof} Let $\mathcal C$ denote the set of conjugates of $w$ and $W$. Note that, by Lemma \ref{<},
the elements 
of $\mathcal C$  are $w$, together with all words $w''$,
and these words are 
distinct since $w$ is primitive. Similarly, since $\mathcal C$ is also the conjugation class of $W$,
the elements of $\mathcal C$ are $W$, together 
with all the words $w'$,
and these words are distinct.

Define a mapping $\nu:\mathcal C\setminus W\to C$ by 
\begin{equation}\label{f}
\nu(w')=w''. 
\end{equation}
This mapping is well-defined, by the previous paragraph, and by Lemma \ref{<}. Moreover, this paragraph also shows that $\nu$ is injective (the words at the right 
of (\ref{f}) are all distinct). Furthermore, by Lemma \ref{<}, $\nu$ maps each element of $\mathcal C\setminus W$ onto a larger element in $\mathcal C$.

Thus by Lemma \ref{next}, $W=\max(\mathcal C)$ and $\nu$ is the next function in $\mathcal C$.  Moreover, $w=\min(\mathcal C)$, since $w$ is not in the image of $\nu$.
\end{proof}

\begin{lemma}\label{position} Let $w$ be a perfectly clustering Lyndon word $w$ with a special palindromic factorization (\ref{special}). Then for any $i$, $a_{i}\pi_i\cdots \pi_{k-1}a_ka_1\pi_1\cdots a_{i-1}\pi_{i-1}$ is the smallest conjugate of $w$ beginning by $a_i$. 
\end{lemma}

\begin{proof} 
The conclusion is clear for $i=1$, since $w$ is a Lyndon word.
Suppose that $i\geq 2$. It is enough to prove that, in the conjugation class of $w$, the word preceding $a_{i}\pi_i\cdots \pi_{k-1}a_ka_1\pi_1\cdots a_{i-1}\pi_{i-1}$ for the lexicographical order begins by $a_{i-1}$. This follows from the next paragraph.

Since $w$ is a perfectly clustering word, it is a product of two palindromes (\cite{SP} Corollary 4.4), hence the word $\tilde w$ is conjugate to $w$; moreover, $\tilde w$ is the largest conjugate of $w$ for lexicographical order (\cite{SP} Theorem 4.3), and clearly $W=\tilde w$. By Lemma \ref{consecutive}, for two consecutive rows of the Burrows-Wheeler matrix of $w$, viewed as two words $w',w''$, one has
$$
w'=va_{i-1}\cdots \pi_1a_1a_k\pi_{k-1}\cdots a_{i}u$$ and $$w''=va_{i}\cdots \pi_{k-1}a_ka_1\pi_1\cdots a_{i-1}u,$$for some $i=1,\ldots,k-1$, with 
$\pi_{i-1}=uv$. Hence $w',w''$ begin by the same letter, except if $v=1$, in which case $w'$ begins by $a_{i-1}$ and $w''=a_{i}\cdots \pi_{k-1}
a_ka_1\pi_1\cdots a_{i-1}\pi_{i-1}$ begins by $a_{i}$.
\end{proof}

\begin{lemma}\label{pi=1} Let $w$ be a perfectly clustering Lyndon word $w$ with a special palindromic factorization (\ref{special}). Let $s=1,\ldots,k-1$; if $\pi_s$ 
is the empty word, then $c_1+\cdots+c_s=c_{s+1}+\cdots+c_k$.
\end{lemma}

\begin{proof} By the remarks following Lemma \ref{order-conjugates}, taking into account Lemma \ref{position}, the hypothesis implies that $\sigma(m_s)=m_{s+1}$, hence we conclude using Lemma \ref{mhmh+1}.
\end{proof}

\begin{lemma}\label{nonempty} Let $w$ be a perfectly clustering Lyndon word $w$ with a special palindromic factorization (\ref{special}). Suppose that $
\rho_\ell(w)$ is a positive word. Let $i$ be maximum such that $a_i\leq \ell$. Then $\pi_j$, $j=i,\ldots,k-1$ are
nonempty.
\end{lemma}

\begin{proof} Since $\rho(u)\in A^*$, we know by Section \ref{basic} that: (*) in $u$, each letter $>\ell$ is 
preceded by some letter $\leq \ell$. 

Suppose first that $j=i+1,\ldots,k-1$.  Then $a_j,a_{j+1}>\ell$. In $u$, $a_{j+1}$ is preceded by $a_j\pi_j$. Hence, by (*), $\pi_j$ cannot be empty.

Suppose now that $j=i$. We have two cases: $a_i=\ell$, $a_i<\ell$. By the result quoted at the end of Section \ref{pcw}, we have $
\sum_{a<\ell}|u|_a>\sum_{a>\ell}|u|_a$. Suppose by contradiction that $\pi_i$ is the empty word; then by Lemma \ref{pi=1}, we have 
$|w|_{a_1}+\cdots+|w|_{a_i}=|w|_{a_{i+1}}+\cdots+|w|_{a_k}$.

Consider the first case: $a_i=\ell$. Then the former inequality becomes $|w|_{a_1}+\cdots+|w|_{a_{i-1}}>|w|_{a_{i+1}}+\cdots+|w|_{a_k}$, contradicting the equality since $|w|_{a_i}>0$.
Consider the second case: $a_i<\ell$, hence $a_{i+1}>\ell$, and $|w|_{\ell}=0$. The inequality becomes 
$|w|_{a_1}+\cdots+|w|_{a_{i}}>|w|_{a_{i+1}}+\cdots+|w|_{a_k}$, contradicting the equality too.
\end{proof}

\section{Proof of Theorem \ref{char1}}\label{proof}

(i) $\Rightarrow$ (ii). It is known that $w$ is conjugate to its reversal $\tilde w$ (see \cite[Corollary 4.4]{SP}). If $w$ is of length at most 2, 
then (ii) holds clearly. Suppose now that $w$ is of length at least 3. Then by the result quoted in Section \ref{basic}, $w=f(u)$ for some 
shorter perfectly clustering word $u$ and some automorphism $f=\lambda_\ell$ or $\rho_\ell$, $\ell\in A$, with the further property: either a) 
$A=Alph(w)=Alph(u)$, or b) $A=Alph(w)=Alph(u)\cup \ell$ and $\ell\notin Alph(u)$.

By induction, we may assume that $u=a_1\pi_1a_2\pi_2\cdots \pi_{k-1}a_k$, 
where $Alph(u)=\{a_1<a_2<\cdots<a_k\}$ and where the word $\pi_i$ are palindromes.
We may assume that $f=\rho_\ell$, the other case being similar. In fact, there is a relation between $\rho_\ell$ and $\lambda_\ell$ given by the antimorphism sending the letter $a_{i}$ to $a_{k-i+1}$. This antimorphism preserve perfectly clustering words (see \cite[Proposition 4.4]{La}).

Write $\rho$ for $\rho_\ell$.
Since $\rho(u)\in A^*$, we know by Section \ref{basic} that: (*) in $u$, each letter $>\ell$ is 
preceded by some letter $\leq \ell$. 

In case a), let $\ell=a_i$. We know that $\pi_{i},\ldots, \pi_k$ are nonempty, by Lemma \ref{nonempty}.

We have $u=\left(\prod_{1\leq j\leq i-1}a_j\pi_j\right) a_i \left(\prod_{i\leq j\leq k-1}\pi_ja_{j+1}\right)$. Hence $$w=\left(\prod_{1\leq j\leq i-1}
a_ja_i\rho(\pi_j)\right) a_i \left(\prod_{i\leq j\leq k-1}\rho(\pi_j)a_i^{-1}a_{j+1}\right).$$ Since $\pi_j$ is nonempty for $j=i,\ldots,k-1$, and is followed by $a_{j+1}$ in $u$,
it follows from Lemma \ref{positive-pal} that $\rho(\pi_j)a_i^{-1}$ is a positive palindrome. For $j=1,\ldots,i-1$, $\rho(\ell \pi_i)=\rho(a_i\pi_i)=a_i\rho(\pi_j)$ is a positive palindrome by (*) and 
Lemma \ref{pal}.
This concludes case a).

In case b), we have for some $i$, $a_1<\cdots<a_{i}<\ell<a_{i+1}<\ldots<a_k$. 

We know that $\pi_{i+1},\ldots, \pi_k$ are nonempty, by Lemma \ref{nonempty}.

Note that $i\geq 1$, since otherwise $\ell<a_1$, and since 
$a_1$ is not preceded in $u$ by any letter, this contradicts (*). 

We have therefore
$$u=a_1\prod_{1\leq j\leq k-1}\pi_ja_{j+1}=
a_1\left(\prod_{1\leq j\leq i-1}\pi_ja_{j+1}\right)\left(\prod_{i\leq j\leq k-1}\pi_ja_{j+1}\right),$$ hence 
\begin{align*}
    w &=a_1\ell\left(\prod_{1\leq j\leq i-1}\rho(\pi_j)a_{j+1}\ell\right)
\left(\prod_{i\leq j\leq k-1}\rho(\pi_j)\ell^{-1} a_{j+1}\right) \\
 &=a_1\left(\prod_{1\leq j\leq i-1}\ell\rho(\pi_j)a_{j+1}\right)\ell \left(\prod_{i\leq j\leq k-1}\rho(\pi_j)\ell^{-1} a_{j+1}\right)\\
 &=a_1q_1\cdots a_{i-1}q_{i-1}a_iq_i\ell q_{i+1}\cdots a_{k-1}q_{k}a_k,
\end{align*}
 with $q_1=\ell\rho(\pi_1),\ldots,q_{i-1}=
\ell\rho(\pi_{i-1}),q_i=1,q_{i+1}=\rho(\pi_i)\ell^{-1},\ldots,q_k=\rho(\pi_{k-1})\ell^{-1}$.
The elements $q_j$ are all positive palindromes: this is proved as in case a).

(ii) $\Rightarrow$ (iii). By hypothesis $w$ has a palindromic special factorization (\ref{special}) and $w$ is conjugate to $\tilde w$, since $w$ is a product of two palindromes. Since the $\pi_i$ are palindromes, $\tilde w=W$. Thus (iii) holds.

(iii) $\Rightarrow$ (i) 
Let $\mathcal C$ denote the set of conjugates of $w$. By hypothesis, $W\in\mathcal C$. Define a mapping $\nu:C\setminus W\to C$ by 
\begin{equation}\label{f1}
\nu(va_i\cdots \pi_1a_1a_k\pi_{k-1}\cdots a_{i+1}u)=va_{i+1}\cdots \pi_{k-1}a_ka_1\pi_1\cdots a_iu, 
\end{equation}
where $\pi_i=uv$.
This mapping is the next element function, by Lemma \ref{consecutive}, 
$W$ is the largest element of $\mathcal C$, and $w$ its smallest, and in 
particular a 
Lyndon word. Thus $\mathcal C=\{w<\nu(w)<\nu^2(w)<\cdots<\nu^{n-1}(w)\}$ with $n$ the length of $w$.

Now, by Equation (\ref{f1}), we see that the last letter of $x=va_i\cdots \pi_1a_1a_k\pi_{k-1}\cdots a_{i+1}u$ and that of $\nu(x)$ are equal, except if $u$ is the empty word. 
This can happen exactly for $k-1$ elements $x$ of $C\setminus W$, namely those of the form $x_i=\pi_ia_i\cdots \pi_1a_1a_k\pi_{k-1}\cdots a_{i+1}$ ($u=1$, 
$v=\pi_i$). The last letter of $x_i$ is $a_{i+1}$ and that of $\nu(x_i)=\pi_ia_{i+1}\cdots \pi_{k-1}a_ka_1\pi_1\cdots a_i$ is $a_i$. Since the last letter of $w$ is 
$a_k$, this implies that the last letters of the words $w,\nu(w),\nu^2(w),\cdots,\nu^{n-1}(w)$ form a weakly decreasing sequence. Thus the last column of the 
Burrows-Wheeler matrix of $w$ is weakly decreasing, and $w$ is therefore perfectly clustering.

\section{Consequences}

From Lemma \ref{position}, we may deduce the uniqueness of the special palindromic factorization of a perfectly clustering Lyndon word, whose existence is asserted in Theorem \ref{char1}.

\begin{corollary} The factorization (\ref{special}) of a perfectly clustering Lyndon word $w$ is unique. Precisely: for any $i$, $a_i\pi_i\cdots \pi_{k-1}a_k$ is the smallest suffix of $w$ beginning by $a_i$, and $a_{i}\pi_i\cdots \pi_{k-1}a_ka_1\pi_1\cdots a_{i-1}\pi_{i-1}$ is the smallest conjugate of $w$ beginning by $a_i$.
\end{corollary}

\begin{proof} 
    The fact that the two last assertions are equivalent follows for example from \cite{HR} Proposition 2.1.
\end{proof}

We study now the transition from a row to the next row, in the Burrows-Wheeler matrix of a perfectly clustering word. We recall first what happens in the case of two letters.

Consider an alphabet $A=\{a<b\}$ with two-letters. A perfectly clustering Lyndon word $w$ on $A^*$ is a lower Christoffel word, by a theorem of Mantaci, Restivo 
and Sciortino \cite[Theorem 9]{MRS}. Thus the last column of the Burrows-Wheeler matrix of $w$ is of the form $b,\ldots,b,a,\ldots,a$. A further striking property is 
that two consecutive rows of this matrix, viewed as two words $u,v$, are always of the form $u=yabx,v=ybax$, with moreover $xy=p$, where $p$ is the 
palindrome such that $w=apb$; see \cite[Corollary 5.1]{BR} (see also \cite[Theorem 2]{PR}, or \cite[Theorem 15.2.4]{R}). For an example, see Figure \ref{BWmatrix}, 
where $w=aaabaab$, $p=aabaa$ and for example the rows 3 and 4: $u=a{\bf ab}aaba,v=a{\bf ba}aaba$, $x=aaba, y=a$.

\begin{figure} 
$$
\begin{array}{cccccccccc}
a&a&a&b&a&a&b\\
a&a&b&a&a&a&b\\
a&a&b&a&a&b&a\\
a&b&a&a&a&b&a\\
a&b&a&a&b&a&a\\
b&a&a&a&b&a&a\\
b&a&a&b&a&a&a
\end{array}
$$
\caption{Burrows-Wheeler matrix of the Christoffel word $aaabaab$}\label{BWmatrix}
\end{figure}

We may generalize this result as follows.

\begin{theorem} Let $w$ be a perfectly clustering word. Then for any two consecutive rows of its Burrows-Wheeler matrix, viewed as two words $w',w''$, one has $w'=ymx,w''=y\tilde mx$, for some words $x,y,m$ such that $xy$ is one of the palindromes of the special palindromic factorization (\ref{special}) of $w$.
\end{theorem}

See Figure \ref{BWmatrix2} for an example: the alphabet is $a<b<c$, $w={\bf a}cacac{\bf b}bb{\bf c}$ (special palindromic factorization, with palindromes $cacac$ and $bb$); rows 1 and 2 are $acac{\bf acbbb}c, acac{\bf bbbca}c$ with the reversed factor in bold, and corresponding palindrome $cacac$; rows 4 and 5 are $bb{\bf bcacacac}, bb{\bf cacacacb}$ and palindrome $bb$.

\begin{figure} 
$$
\begin{array}{cccccccccc}
a&c&a&c&a&c&b&b&b&c\\
a&c&a&c&b&b&b&c&a&c\\
a&c&b&b&b&c&a&c&a&c\\
b&b&b&c&a&c&a&c&a&c\\
b&b&c&a&c&a&c&a&c&b\\
b&c&a&c&a&c&a&c&b&b\\
c&a&c&a&c&a&c&b&b&b\\
c&a&c&a&c&b&b&b&c&a\\
c&a&c&b&b&b&c&a&c&a\\
c&b&b&b&c&a&c&a&c&a
\end{array}
$$
\caption{Burrows-Wheeler matrix of the perfectly clustering Lyndon word $acacacbbbc$}\label{BWmatrix2}
\end{figure}

\begin{proof} By Lemma \ref{consecutive}, $w'=va_i\cdots \pi_1a_1a_k\pi_{k-1}\cdots a_{i+1}u$ and $w''=va_{i+1}\cdots \pi_{k-1}a_ka_1\pi_1\cdots a_iu$, for 
some $i=1,\ldots,k-1$, with $\pi_i=uv$. Then the theorem follows by taking $y=v,x=u,m=a_i\cdots \pi_1a_1a_k\pi_{k-1}\cdots a_{i+1}$, so that $\tilde 
m=a_{i+1}\cdots \pi_{k-1}a_ka_1\pi_1\cdots a_i$, since the $\pi_j$ are palindromes.
\end{proof}

Since the arguments in the proof of Lemma \ref{pi=1} may be reversed, we obtain the following corollary.

\begin{corollary} Let $w$ be a perfectly clustering word, with its special palindromic factorization \ref{special}. Then $\pi_s$ is empty if and only if $|w|_{a_1}+\cdots+|w|_{a_s}=|w|_{a_{s+1}}+\cdots+|w|_{a_k}$.
\end{corollary}

This corollary is a first result on the palindromes appearing in factorization (\ref{special}). We intend to study these palindromes elsewhere. Note that for a two-letter alphabet, these palindromes, called {\em central words}, where first studied by Aldo de Luca. He gave, among others, the following beautiful characterization: the central words are the image of the mapping called {\em iterated palindromization} (see e.g. \cite[Chapter 12]{R}).

\end{document}